\documentclass[a4paper, reqno, 12pt]{amsart}

\usepackage[T1]{fontenc}
\usepackage[utf8]{inputenc}
\usepackage{amsmath}
\usepackage{amssymb}
\usepackage{array}
\usepackage[backend=biber,style=alphabetic]{biblatex}
\usepackage{colonequals}
\usepackage[margin=3cm]{geometry}
\usepackage{hyperref}
\usepackage{mathtools}
\usepackage{pgfplots}

\DeclareFieldFormat{doi}{\newline\small\textsc{doi: }\href{https://doi.org/#1}{\nolinkurl{#1}}}
\DeclareFieldFormat{url}{\newline\small\textsc{url: }\href{https://#1}{\nolinkurl{#1}}}

\numberwithin{equation}{section}
\subjclass[2020]{11Y05, 11Y16, 68W10}
\pgfplotsset{compat=1.18}

\title{Choosing iteration maps for the parallel Pollard rho method}
\date{\today}
\author{Finn Rudolph}

\newtheorem{lemma}{Lemma}
\newtheorem{theorem}{Theorem}
\newtheorem{heuristic}{Heuristic}

\newcommand{\N}{\mathbb N}
\newcommand{\Z}{\mathbb Z}
\newcommand{\R}{\mathbb R}
\renewcommand{\phi}{\varphi}
\DeclareMathOperator{\lcm}{lcm}

\DeclarePairedDelimiter{\paren}{\lparen}{\rparen}
\DeclarePairedDelimiter{\curly}{\lbrace}{\rbrace}

\DeclarePairedDelimiter{\vertbar}{\lvert}{\rvert}

\newcommand{\E}{\mathbb E}
\newcommand{\F}{\mathcal F}
\newcommand{\td}{\widetilde}

\newcommand{\tfc}{N}
\newcommand{\nd}{n}
\newcommand{\eff}{q}
\newcommand{\idg}{d}
\newcommand{\tf}{\lambda}
\newcommand{\st}{s}

\newcommand{\stv}{S}
\newcommand{\ryl}{Q}
\newcommand{\tv}{X}
\newcommand{\grt}{G}
\newcommand{\rya}{T}
\newcommand{\lcmki}{\ell}
\newcommand{\fns}{\mathcal F}
\newcommand{\itfn}{f}

\begin{document}

\begin{abstract}
  Pollard's rho method finds a prime factor $p$ of an integer $\tfc$ by
  searching for a collision in a map of the form $x \mapsto x^{2k} + c$
  modulo $\tfc$. This search can be parallelized to multiple
  machines, which may use distinct parameters $k$ and $c$. In this
  paper, we give an asymptotic estimate for the expected running time
  of the parallel rho method depending on the choice of $k$ for each
  machine. We also prove
  that $k = 1$ is the best choice for one machine, if nothing about
  $p$ is known in advance.
\end{abstract}

\maketitle

\section{Introduction}

In \cite[p.~255,5.24]{cp05}, Crandall and Pomerance pose the following question:
Assume $M$ machines run the rho method independently on an integer
$\tfc$ and the $i$-th machine uses $x \mapsto x^{2k_i} + c_i$ as its
iteration map. How to assign parameters $k_i$ to them in an optimal
way? This problem is the subject of this paper. Let
$\tfc \in \N_{\ge 2}$ be the number to be factored and $p$ be the
smallest prime divisor of $\tfc$ not equal to 2. Statements about the
expected running time of the rho method to factor $\tfc$ are always
made in terms of $p$.

The choice of the $c_i$ is not discussed here, since they do not seem
to have much influence on the
running time and are best chosen independently at random. (Except
  that $0$ and $-2$ should be avoided, see \cite[p.~333]{pol75} and
\cite[p.~6]{cr99}.)
In contrast, the $k_i$ can affect
the running time significantly depending on their common factors with
$p - 1$. Let $d_i \colonequals \gcd(p - 1, 2k_i)$. The expected number
of steps the $i$-th machine requires is reduced
by a factor of $\sqrt{d_i - 1}$ (this will become clear from Theorem
\ref{th:main}, or see \cite{bp81}). However, the
running time per iteration of the $i$-th machine is increased by a
factor of $\log_2 2k_i$ in comparison to $k_i = 1$. Thus, it is not
immediately clear which choice of $k_i$ yields the best expected
running time.

\subsection{The main result}
\label{sec:main-result}

As usual for the analysis of the rho method, we assume that
the rho iteration map behaves like a random map (see Heuristic
\ref{heu:rma}). Under this assumption, the expected time until some
machine finds $p$ as a factor of $\tfc$ is asymptotic to
\begin{align}
  \sqrt {\pi p / 2} \, \paren*{ \sum_{i = 1}^M
    \frac {\gcd(p - 1, 2k_i) - 1} {\log_2^2 2k_i}
  }^{\!\! -1/2}
  \label{eq:main}
\end{align}
as $p \to \infty$. We say two functions $f, g : (0, \infty) \to (0,
\infty)$ are \emph{asymptotic} and write $f \sim g$, if $\lim_{x \to
\infty} f(x)/g(x) = 1$.

For $M = k_1 = 1$, (\ref{eq:main}) evaluates to
$\sqrt{\pi p / 2}$, which is the
classical estimate for rho method given
by \cite{pol75}. For $M = 1$ and arbitrary $k_1$, (\ref{eq:main}) reduces to
\begin{align*}
  \sqrt{\pi p / 2} \; \frac {\log_2 2k_1} {\sqrt{\gcd(p - 1, 2k_1) - 1}} \,.
\end{align*}
This also matches previous results, see e.g. \cite{bp81} or
\cite[p.~255,5.24]{cp05}.

\subsection{Formulation for random maps}
\label{sec:general-formulation}

We briefly recall aspects of the rho method relevant to our
analysis. For a complete description of the algorithm, see
\cite{pol75} or \cite[p.~229]{cp05}. As mentioned in
§\ref{sec:main-result}, there is a commonly used heuristic for
analyzing the rho method. To state it, a few definitions are required. Let
\begin{align*}
  \itfn_i: \Z/p\Z \to \Z/p\Z, \quad \itfn_i : x \mapsto x^{2k_i} + c_i
\end{align*}
be the iteration map of the $i$-th machine, viewed modulo $p$, and
$x_{i,0} \in \Z/p\Z$ its initial value, also modulo $p$. The initial
value is chosen uniformly at random from $\Z/\tfc\Z$.

Let $f : X \to X$ be a function on a finite set $X$ and denote by $f^j$
the $j$-fold composition of $f$ with
itself. For an element $x \in X$, we define the \emph{rho
length} of $f$ at $x$ by
\begin{align*}
  \rho(f, x) \colonequals \vertbar*{ \curly*{ f^j(x) : j \in \N} }.
\end{align*}
We include 0 in the natural numbers, so $x$ is also in the above set.

Let $n, d \in \N_{\ge 1}$ and assume that $d$ divides $n$. Define
$\F_\nd^\idg$ as the set of all functions $f : \Z/\nd\Z \to \Z/\nd\Z$
such that $|f^{-1}\{x\}| \in \{0, \idg\}$ for every $x \in \Z/\nd\Z$.
When picturing an element of $\F_{\nd}^{\idg}$ as a directed graph,
exactly $\nd/\idg$ nodes have indegree $\idg$, and all other nodes have
indegree 0.

The main issue in the analysis of the rho method is to estimate the
number of iterations a machine requires to discover $p$. This number
is $\rho(\itfn_i, x_{i,0})$ for the the $i$-th
machine, since it iterates $\itfn_i$ starting from $x_{i, 0}$ until a
value is encountered twice. However, there are currently no estimates
for the distribution or mean of $\rho(\itfn_i, \, \cdot \,)$ known,
so we rely on the following heuristic.

\begin{heuristic}
  \label{heu:rma}
  The distribution of $\rho(\itfn_i, \, \cdot \,)$ is close to the
  distribution of the map $\F_{p - 1}^{\idg_i} \times \Z/(p - 1)\Z
  \to \N$ sending $(f, x)$ to $\rho(f, x)$.
\end{heuristic}

\emph{Close} is to be understood as \emph{close enough to be
interchangeable for the analysis of the rho method.} An intuitive
justification for Heuristic \ref{heu:rma} stems from the distribution
of indegrees of $\itfn_i$. Let $(\Z/p\Z)^\times$ be the group of units of
$\Z/p\Z$. Since $(\Z/p\Z)^\times$ is cyclic, exactly $(p -
1)/\idg_i$ elements have $\idg_i$ preimages under the map $x \mapsto
x^{2k_i}$ on $(\Z/p\Z)^\times$. So the distribution of indegrees of
$\itfn_i$ equals the distribution of indegrees of elements in $\F_{p
- 1}^{\idg_i}$, up to a shift by one due to the inclusion of zero.
Thus, heuristically assuming that $\itfn_i$ exhibits no further patterns, we
approximate the distribution of $\rho(\itfn_i, \,\cdot\,)$ by the
rho length distribution of a random element of $\F_{p - 1}^{\idg_i}$.
Crude as this heuristic may seem, it works remarkably well in
practice. Theoretical and numerical results regarding Heuristic
\ref{heu:rma} are summarized in §\ref{sec:rma}.

We now formulate a rigorous result that can be used to derive
(\ref{eq:main}) under Heuristic \ref{heu:rma}. Let $n \in
\N_{\ge 1}$ and $d_i \in \N_{\ge 1}$ such that $d_i \mid n$ for every $1 \le
i \le M$. We consider $M$ machines simultaneously iterating a function on
$\Z/n\Z$. Each machine stops when it encounters a repeated value. The
function used by the $i$-th machine is drawn uniformly at
random from $\F_{n}^{\idg_i}$ and its initial value is drawn
uniformly from $\Z/n\Z$. The setup is represented by the space
$\Omega \colonequals \prod_{i = 1}^M \paren*{
  \Z/\nd\Z \times \mathcal F_{\nd}^{\idg_i}
}$,
equipped with the uniform probability measure.

\begin{theorem}
  For $1 \le i \le M$, let $\lambda_i \in (0, \infty)$ and define random
  variables $\tv_i : \Omega \to \R$ by $\tv_i((x_1, g_1), \dots,
  (x_M, g_M)) = \tf_i \rho(g_i, x_i)$. Then
  \begin{align}
    \E \paren*{ \min_{i = 1}^M \tv_i }
    \sim \sqrt{\pi \nd / 2}
    \paren*{
      \sum_{i = 1}^M \frac{\idg_i - 1}{\tf_i^2}
    }^{\!\!-1/2}
    \label{eq:general}
  \end{align}
  as $n \to \infty$.
  \label{th:main}
\end{theorem}

One may interpret $\lambda_i$ as the time per iteration of the $i$-th
machine, and $\tv_i$ as the time required by the $i$-th machine to
terminate. Formula (\ref{eq:main}) follows from Theorem \ref{th:main}
by taking $\nd = p - 1$, $\idg_i = \gcd(p - 1, 2k_i)$ and $\tf_i =
\log_2 2k_i$. We choose $\tf_i = \log_2 2k_i$, because calculating
the power $x^{2k_i}$ in each iteration requires $\log_2 2k_i$ steps
(see §\ref{sec:time-per-it} for more details).

\subsection{Without information about \emph{p}}

To apply (\ref{eq:main}), we need to know $\gcd(p - 1, 2k_i)$ for
every $1 \le i \le M$. But often such information is not available
when trying to factor a large number. This raises a natural question:
Which $k_i$ should be chosen if nothing about $p$ is known?

To address this question, we consider all possible congruences of $p$
modulo $\lcmki \colonequals \lcm \{2k_i : 1 \le i \le M\}$, since we
already know the expected running time for a fixed congruence
class from (\ref{eq:main}). Since $p$
is prime, the possible congruence classes are
exactly the units in $\Z/\lcmki\Z$ (we may assume $p \nmid \lcmki$,
since $p$ is much larger than $\lcmki$ in all cases of interest).
Dirichlet's theorem implies that all these congruence classes occur
with equal frequency for sufficiently large $p$ (see e.g.
\cite[p.~251, Theorem 1]{ir90}). So we define
\begin{align}
  \grt(k_1, \dots, k_M) \colonequals
  \frac 1 {\phi(\lcmki)}
  \sum_{d \in \paren* {\Z/\lcmki\Z}^\times}
  \paren*{ \sum_{i = 1}^M
    \frac {\gcd(d - 1, 2k_i) - 1} {\log_2^2 2k_i}
  }^{\!\! -1/2}.
  \label{eq:def-grt}
\end{align}
By the above discussion, $\grt(k_1, \dots, k_M)
\sqrt{\pi p /2}$ can be thought of as the expected time to find the
prime factor $p$, when no information about $p$ is available. The
term $\sqrt{\pi p / 2}$ of (\ref{eq:main}) was omitted in
(\ref{eq:def-grt}), because it is irrelevant for comparing different
parameterizations.

The choice of $k_i$ that minimizes $\grt(k_1, \dots, k_M)$ could not be
identified for general $M$. Numerical computations for small $M$ and
$k_i$ suggest that $k_1 = \cdots = k_M = 1$ is optimal for every $M
\ge 1$ (see §\ref{sec:table-of-values} and
§\ref{sec:opt-k-general-M}). For $M = 1$, we prove this rigorously.

\begin{theorem}
  \label{th:M1}
  $\grt(1) < \grt(k)$ for every $k \in \N_{\ge 2}$.
\end{theorem}

\subsection{Structure of the paper}

Section §\ref{sec:prelim} introduces the notation and definitions,
primarily for the proof of Theorem \ref{th:main}. The proof of
Theorem \ref{th:main} is carried out in §\ref{sec:int-approx} and
§\ref{sec:ryl-approx}. In §\ref{sec:opt-k-M1}, we prove Theorem
\ref{th:M1}. Section §\ref{sec:table-of-values} presents computed
values of $G(k_1, \dots, k_M)$ and discusses their behavior. We
conclude with final remarks in §\ref{sec:final-remarks}.

\subsection*{Acknowledgements}

I thank Yannick Sommer for a helpful discussion about the proof of
Theorem \ref{th:M1}.

\section{Preliminaries}
\label{sec:prelim}

As in §\ref{sec:general-formulation}, assume that $M, \nd$ and
$\idg_i$ are nonzero natural numbers such that $\idg_i \mid n$ for $i =
1, \dots, M$. Define $\Omega \colonequals \prod_{i =
1}^M \paren*{ \Z/\nd\Z \times \fns_{\nd}^{\idg_i} }$
with the uniform probability measure $P$. On $\Omega$, let $\stv_i :
\Omega \to \N$ be random variables given by
\begin{align}
  \stv_i : ((x_1, g_1), \dots, (x_M, g_M)) \mapsto \rho(g_i, x_i) \, .
  \label{eq:def-rand-var}
\end{align}
Like in Theorem \ref{th:main}, we also let $\tf_i \in (0, \infty)$ for
$1 \le i \le M$ and define $\tv_i = \tf_i \stv_i$.

Let $\eff_i \colonequals \nd/(\idg_i - 1)$. Brent and Pollard showed
in \cite{bp81} that
\begin{align*}
  P(\stv_i = \st \, | \, \stv_i \ge \st) = \st(\idg_i - 1)/\nd = \st/\eff_i \,.
\end{align*}
Thus
\begin{align}
  P(\stv_i = \st) =
  \begin{cases}
    \frac {\st} {\eff_i} \prod_{j = 0}^{\st - 1} \paren*{
      1 - \frac j {\eff_i}
    } & \st \le  \eff_i \\
    0 & \st > \eff_i \\
  \end{cases} \, ,
  \label{eq:collision-probability}
\end{align}
because in the first $\st - 1$ steps, no collision may occur, and in
the last step, a collision must occur.

Using the definitions in (\ref{eq:def-rand-var}) and
(\ref{eq:collision-probability}), the left hand side of
(\ref{eq:general}) can be expressed as
\begin{align}
  \E \paren*{\min_{i = 1}^M \tv_i }
  = \sum_{\st_1 = 0}^\infty P(\stv_1 = \st_1) \, \cdots \!
  \sum_{\st_M = 0}^\infty P(\stv_M = \st_M) \min_{i = 1}^M (\tf_i \st_i) \, .
  \label{eq:expectation-expr}
\end{align}
The dots in the middle represent the sums over $\st_i$ for $i = 2,
\dots, M - 1$.

As a motivation for the next definition, we approximate $1 -
j/\eff_i$ by $e^{-j/\eff_i}$ in (\ref{eq:collision-probability}). This yields
\begin{align*}
  P(\stv_i = \st) \approx \frac {\st} {\eff_i} e^{-\st(\st - 1)/(2\eff_i)}
  \approx \frac {\st} {\eff_i} e^{-\st^2/(2\eff_i)} \, .
\end{align*}
The right expression is more manageable than the product in
(\ref{eq:collision-probability}) and allows us to find an asymptotic
estimate for (\ref{eq:expectation-expr}).
Section §\ref{sec:ryl-approx} is devoted to proving that this
approximation can be applied to (\ref{eq:expectation-expr}) in a
rigorous way. For $\eff \in (0, \infty)$, define $\ryl_\eff :
(0, \infty) \to \R$ by
\begin{align}
  \ryl_\eff(\st) = \frac \st \eff e^{-\st^2/(2\eff)}.
  \label{eq:def-q}
\end{align}
Note that $\ryl_{\eff}$ is the density function of a
Rayleigh-distributed random variable with parameter $\sqrt{\eff}$.

\section{Asymptotic estimate of the Rayleigh-approximated version}
\label{sec:int-approx}

The proof of Theorem \ref{th:main} is structured as follows: To
estimate (\ref{eq:expectation-expr}) asymptotically, we first
replace $P(\stv_i = \st_i)$ by the Rayleigh densities
$\ryl_{\eff_i}(\st_i)$, yielding
\begin{align}
  \rya_{\idg_1, \dots, \idg_M, \tf_1, \dots, \tf_M}(\nd) \colonequals
  \sum_{\st_1 = 0}^{\infty} \ryl_{\eff_1}(\st_1) \, \cdots \!
  \sum_{\st_M = 0}^{\infty} \ryl_{\eff_M}(\st_M) \min_{i = 1}^M (\tf_i \st_i)
  \, .
  \label{eq:def-ryl-approximated}
\end{align}
When unambiguous in the context, we abbreviate $\rya_{\idg_1, \dots,
\idg_M, \tf_1, \dots, \tf_M}(\nd)$ by $\rya_{\idg_i,
\tf_i}(\nd)$. In this section, we derive an asymptotic estimate for
$\rya_{\idg_i, \tf_i}$, which will be exactly the right hand side of
(\ref{eq:general}). In §\ref{sec:ryl-approx}, we show that
\begin{align*}
  \E \paren*{\min_{i = 1}^M \tv_i } \sim \rya_{\idg_i, \tf_i} \,,
\end{align*}
thereby completing the proof of Theorem \ref{th:main}.

To evaluate $\rya_{\idg_i, \tf_i}$, we wish to replace the sums by
integrals. In Lemma \ref{lemma:sum-integral}, the
Euler-Maclaurin formula is applied to show that this substitution
preserves the asymptotic behavior in the case $M = 1$. We allow a
slightly more general
class of functions than necessary, which allows us to extend the
result to $M \ge 2$ later. For $a \in \N$ and $b \in \N \cup
\{\infty\}$, define $I(a, b) = [a, b]$ if $b < \infty$ and $I(a, b) =
[a, b)$ if $b = \infty$.

\begin{lemma}
  Let $a \in \N$ and $b \in \N \cup \{\infty\}$ such that $a \le b$.
  Let $g : I(a, b) \to \R$ be a piecewise linear function, which is
  differentiable at all but finitely many points.  Then
  \begin{align}
    \sum_{\st = a}^b g(\st) \ryl_\eff(\st)
    - \int_a^b g(\st)\ryl_\eff(\st) \, d\st
    = O \paren*{ 1 }
    \label{eq:lemma-eumac}
  \end{align}
  as $\eff \to \infty$. For fixed $g$, this holds
  uniformly with respect to $a$ and $b$. That is, there exist two
  $O$-constants that work
  for every pair $a \le b$.
  \label{lemma:sum-integral}
\end{lemma}

\begin{proof}
  We make two quick reductions. First, we may assume that $b <
  \infty$, because we will prove that
  the asymptotic estimate holds uniformly with respect to $a, b \in
  \N$. Thus it also holds in the limit $b \to \infty$.

  Second, we reduce to the case that $g$ is linear. For this, let
  $u_1 < \cdots < u_l$ be the non-differentiable points of $g$ in
  $(a, b)$. Define $u_0 \colonequals a$ and $u_{l + 1} \colonequals b + 1$. Then
  \begin{align*}
    &\sum_{\st = a}^b g(\st) \ryl_\eff(\st) - \int_a^b g(\st)\ryl_\eff(\st) \\
    &= \sum_{j = 0}^l \paren*{
      \smashoperator[r]{
        \sum_{s = \lceil u_j \rceil}^{\lceil{u_{j + 1}} \rceil - 1}
      }
      \ g(\st) \ryl_\eff(\st) \,
      - \, \smashoperator[lr]{
        \int_{\lceil u_j \rceil}^{\lceil{u_{j + 1}} \rceil - 1}
      }
    g(\st) \ryl_\eff(\st) \, d\st }
    - \sum_{j = 1}^l \! \! \!
    \smashoperator[r]{\int_{\ \ \, \lceil u_j \rceil - 1}^{\lceil{u_j} \rceil }}
    g(\st) \ryl_\eff(\st) \, d\st \,.
  \end{align*}
  Because $g$ is linear on every interval $[\lceil u_j \rceil,
  \lceil{u_{j + 1}} \rceil - 1]$, the left sum in the
  second row is bounded by $O(\eff^{-1/2})$ if (\ref{eq:lemma-eumac})
  holds for linear functions.
  The right sum is $O(\eff^{-1})$, because $\ryl_\eff(\st) =
  O(\eff^{-1})$ as $\eff \to \infty$ for every $\st \in [0, \infty)$.
  Thus, it suffices to prove (\ref{eq:lemma-eumac}) in the
  case that $g$ is linear.

  For the sake of shorter notation, let $h(\st) \colonequals
  g(\st)\ryl_\eff(\st)$. By
  applying the second derivative version of the Euler-Maclaurin formula, we get
  \begin{align}
    \sum_{\st = a}^b h(\st)
    = \int_a^b h(\st) \, d\st
    + \frac {h(a) + h(b)} 2
    + \frac {B_2} 2 \paren*{h'(b) - h'(a) }
    + R_2 \,,
    \label{eq:euler-maclaurin}
  \end{align}
  where $B_2$ is the second Bernoulli number and
  $R_2$ the remainder term. The second summand of
  (\ref{eq:euler-maclaurin}) expands to
  \begin{align}
    \frac {h(a) + h(b)} 2
    = \frac {g(a) ae^{-a^2/(2\eff)} + g(b) b e^{-b^2/(2\eff)}} {2\eff} \,.
    \label{eq:ha-hb}
  \end{align}
  Note that $\st \mapsto (\st^2/\eff) e^{-\st^2/(2\eff)}$ is bounded
  for $\st \in [0, \infty)$ and its maximal value is independent of
  $\eff$. Thus (\ref{eq:ha-hb}) is bounded independently of
  $\eff, a$ and $b$ for $q \ge 1$. As for the third summand of
  (\ref{eq:euler-maclaurin}), the derivative of $h$ is
  \begin{align*}
    h'(\st)
    = \paren*{\frac {\st g'(\st) + g(\st)} \eff
    - \frac {\st^2g(\st)} {\eff^2} } e^{-\st^2/(2\eff)} \,.
  \end{align*}
  Again using that $\st \mapsto (\st^2/\eff) e^{-\st^2/(2\eff)}$ is
  bounded independently of $\eff$, we see that $h'(a)$ and $h'(b)$
  are bounded independently of $\eff, a$ and $b$. Thus the third
  summand of (\ref{eq:euler-maclaurin}) is $O \paren*{ 1 }$.

  Lastly, the remainder term is
  \begin{align}
    R_2 = \int_a^b \frac {B_2(\st - \lfloor \st \rfloor)} 2
    \, h''(\st) \, d\st \,.
    \label{eq:eumac-remainder}
  \end{align}
  where $B_2$ is the second Bernoulli polynomial. The second
  derivative of $h$ is
  \begin{align*}
    h''(\st)
    &= \paren*{
      \frac {2g'(\st)} {\eff}
      - \frac {\st^2g'(\st)} {\eff^3}
      + \frac {\st^3g(\st)} {\eff^3}
      - \frac {3\st g(\st)} {\eff^2}
    }
    e^{-\st^2/(2\eff)} \,.
  \end{align*}
  Inserting this into (\ref{eq:eumac-remainder}) and
  applying the fact that $|B_2(\st)| \le 1$ for every $\st \in [0, 1]$, we get
  \begin{align}
    \vertbar*{R_2} \le
    \int_a^b
    \paren*{
      \vertbar*{ \frac {2g'(\st)} {\eff} }
      + \vertbar*{ \frac {\st^2g'(\st)} {\eff^3} }
      + \vertbar*{ \frac {\st^3g(\st)} {\eff^3} }
      + \vertbar*{ \frac {3\st g(\st)} {\eff^2} }
    }
    e^{-\st^2/(2\eff)} \, d\st \,.
    \label{eq:emac-remainder-bound}
  \end{align}
  We may assume that $g(\st) \ge 0$ and $g'(\st) \ge 0$ for every $\st \ge 0$.
  If this is not the case, we replace $g$ by a different linear function $\td
  g$ such that $\td g(\st) \ge |g(\st)|$ for $\st \ge 0$. Thus the absolute
  value bars in (\ref{eq:emac-remainder-bound}) may be removed.
  By integrating the four summands separately and using that
  $\int_0^\infty \st^l e^{-\st^2/(2\eff)} \, d\st = O \paren*{
  \eff^{(l + 1)/2} }$, it follows that $|R_2| = O \paren*{\eff^{-1/2}}$.
\end{proof}

Repeatedly applying Lemma \ref{lemma:sum-integral} to the nested sum in
(\ref{eq:def-ryl-approximated}) allows us to replace each sum by an integral.
This yields the following estimate for $\rya_{\idg_i, \tf_i}$.

\begin{lemma}
  As $n \to \infty$,
  \begin{align}
    \rya_{\idg_i,\tf_i}(\nd) = \sqrt{\pi \nd / 2} \, \paren*{ \sum_{i
      = 1}^M \frac
    {\idg_i - 1} {\tf_i^2} }^{\!\! -1/2} + O \paren*{ 1 } \,.
    \label{eq:ryl-approximated-formula}
  \end{align}
  \label{lemma:ryl-approximated-formula}
\end{lemma}

\begin{proof}
  Let $g : [0, \infty)^M \to \R$ be a function which is piecewise linear with
  finitely many non-differentiable points when any $M - 1$
  coordinates are fixed. We claim that for such $g$
  \begin{align}
    \begin{split}
      \sum_{\st_1 = 0}^\infty \ryl_{\eff_1}(\st_1) \, &\cdots \!
      \sum_{\st_M = 0}^\infty \ryl_{\eff_M}(\st_M)
      \, g(\st_1, \dots \st_M) \\
      &= \int_0^\infty \ryl_{\eff_1}(\st_1) \, \cdots
      \int_0^\infty \ryl_{\eff_M}(\st_M)
      \, g(\st_1, \dots \st_M) \,
      d\st_M \cdots d\st_1
      + O \paren*{ 1 } \,.
    \end{split}
    \label{eq:sum-int-indhyp}
  \end{align}
  The proof is by induction over $M$. For $M = 1$, the claim follows
  immediately from Lemma \ref{lemma:sum-integral}. If $M \ge
  2$, we apply Lemma \ref{lemma:sum-integral} to the innermost sum.
  This is possible because $\st_1, \dots \st_{M - 1}$ are fixed
  inside this sum, so $g$ (as a function in $\st_M$) satisfies the
  hypothesis of Lemma \ref{lemma:sum-integral}. Thus
  \begin{align}
    &\sum_{\st_1 = 0}^\infty \ryl_{\eff_1}(\st_1) \, \cdots \,
    \smashoperator[l]{\sum_{\st_{M - 1} = 0}^\infty}
    \! \ryl_{\eff_{M - 1}}(\st_{M - 1})
    \paren* {
      \int_0^\infty \ryl_{\eff_M}(\st_M)
      \, g(\st_1, \dots, \st_M) \, d\st_M +
      O\paren*{ 1 }
    } \nonumber \\
    &=
    \sum_{\st_1 = 0}^\infty \ryl_{\eff_1}(\st_1) \, \cdots \,
    \smashoperator[l]{\sum_{\st_{M - 1} = 0}^\infty}
    \! \ryl_{\eff_{M - 1}}(\st_{M - 1})
    \int_0^\infty \ryl_{\eff_M}(\st_M)
    \, g(\st_1, \dots, \st_M) \, d\st_M \nonumber \\
    & \quad + \sum_{\st_1 = 0}^\infty \ryl_{\eff_1}(\st_1) \, \cdots \,
    \smashoperator[l]{\sum_{\st_{M - 1} = 0}^\infty}
    \! \ryl_{\eff_{M - 1}}(\st_{M - 1})
    \; O \paren*{ 1 } \, .
    \label{eq:indstep-lemma2-application}
  \end{align}
  The sums in the third line of (\ref{eq:indstep-lemma2-application})
  are independent and may be evaluated separately. By applying Lemma
  \ref{lemma:sum-integral} to each of them, we find that
  \begin{align}
    \sum_{\st_i = 0}^\infty \ryl_{\eff_i}(\st_i)
    = \int_0^\infty \ryl_{\eff_i}(\st_i) \, d\st_i
    + O \paren*{ 1 }
    = O \paren*{ 1 } \, .
    \label{eq:sum-int-err-bound}
  \end{align}
  The integral in (\ref{eq:sum-int-err-bound}) equals 1 because
  $\ryl_{\eff_i}$ is the density of a
  Rayleigh distribution. Thus the third line of
  (\ref{eq:indstep-lemma2-application}) is $O \paren*{1}$. In the second line of
  (\ref{eq:indstep-lemma2-application}), the integral can be moved in
  front of all sums, because they all converge absolutely. Then,
  (\ref{eq:sum-int-indhyp})
  follows from the induction hypothesis.

  By applying this claim to $\min_{i = 1}^M (\st_i \tf_i)$ in
  (\ref{eq:def-ryl-approximated}), we get
  \begin{align*}
    \rya_{\idg_i,\tf_i}(\nd) =
    \int_0^\infty \ryl_{\eff_1}(\st_1) \, \cdots
    \int_0^\infty \ryl_{\eff_M}(\st_M)
    \min_{i = 1}^M(\st_i \tf_i) \, d\st_M \cdots d\st_1
    + O \paren*{ 1 } \,.
  \end{align*}
  A change of variables $t_i \colonequals \st_i \tf_i$ yields
  \begin{align*}
    \rya_{\idg_i,\tf_i}(\nd) =
    \int_0^\infty \ryl_{\eff_1\tf_1^2}(t_1) \, \cdots
    \int_0^\infty \ryl_{\eff_M\tf_M^2}(t_M)
    \min_{i = 1}^M(t_i) \, dt_M \cdots dt_1
    + O \paren*{ 1 } \, .
  \end{align*}
  Thus $\rya_{\idg_i,\tf_i}(\nd)$ is the mean of the minimum of $M$
  Rayleigh-distributed random variables with parameters
  $\sqrt{\eff_i}\tf_i$, up to a constant. The minimum of $M$
  Rayleigh-distributed
  random variables with parameters $\sigma_i$ is again
  Rayleigh-distributed with parameter $\big(\!\sum_{i=1}^M
  1/\sigma_i^2\big)\phantom{}^{-1/2}$. This is a standard fact and
  easy to show via a direct computation. Since the mean of a Rayleigh
  distribution with
  parameter $\sigma$ is $\sigma
  \sqrt{\pi/2}$, it follows that
  \begin{align*}
    \rya_{\idg_i,\tf_i}(\nd)
    &= \sqrt{\pi / 2} \, \paren* {
      \sum_{i = 1}^M \frac {1} {\eff_i\tf_i^2}
    }^{\!\! -1/2}
    + O \paren*{ 1 } \,.
  \end{align*}
  Substituting $\eff_i = n/(\idg_i - 1)$, as defined in
  §\ref{sec:prelim}, we get
  \begin{align*}
    \rya_{\idg_i,\tf_i}(\nd) &= \sqrt{\pi n / 2} \, \paren* {
      \sum_{i = 1}^M \frac {\idg_i - 1} {\tf_i^2}
    }^{\!\! -1/2}
    + O \paren*{ 1 } \,. \qedhere
  \end{align*}
\end{proof}

\section{Justification of the Rayleigh density approximation}
\label{sec:ryl-approx}

For the proof of Theorem \ref{th:main}, it remains to argue that
$\rya_{\idg_i, \tf_i}$ is asymptotic to
(\ref{eq:expectation-expr}). For this purpose, we bound the
difference between $P(\stv_i = \st_i)$ and $\ryl_{\eff_i}(\st_i)$ in
the following Lemma.

\begin{lemma}
  \label{lemma:pq-diff}
  Let $\stv_i$ be as in (\ref{eq:def-rand-var})
  and $\ryl_{\eff_i}(\st)$ as in (\ref{eq:def-q}).
  \begin{enumerate}
    \item[$(a)$] Let $1/2 \le \alpha < 2/3$. Then for every $s \in \N
      \cap [0, \eff_i^\alpha]$,
      \begin{align}
        \vertbar*{\ryl_{\eff_i}(\st) - P(\stv_i = \st)}
        = O\paren*{\eff_i^{4\alpha - 3}}
        \label{eq:P-Q-diff-bound-small}
      \end{align}
      as $\eff_i \to \infty$. This holds uniformly with respect to $\st$.
    \item[$(b)$] For every $\st \in \N$,
      \begin{align}
        \vertbar*{ \ryl_{\eff_i}(\st) - P(\stv_i = \st) }
        \le \frac \st {\eff_i} e^{-(\st - 1)\st/(2\eff_i)} \,.
        \label{eq:P-Q-diff-bound-triv}
      \end{align}
  \end{enumerate}
\end{lemma}

\begin{proof}
  For part $(a)$, let $R(x) \colonequals e^x - (1 + x)$. Then
  \begin{align*}
    \vertbar*{
      \ryl_{\eff_i}(\st) - P(\stv_i = \st)
    } = \frac \st {\eff_i}
    \vertbar*{ e^{-\st^2/(2\eff_i)} -
      \prod_{j=0}^{\st - 1} \paren*{ e^{-j/\eff_i} - R(-j/\eff_i) }
    }\,.
  \end{align*}
  Let $A \colonequals \{0, 1, \dots \st - 1\}$. By expanding the product on the
  right and moving terms with at least one factor of $R$ under separate
  absolute value bars, we get
  \begin{align}
    \begin{split}
      & \vertbar*{ \ryl_{\eff_i}(\st) - P(\stv_i = \st) } \\
      & \le \frac \st {\eff_i}
      \vertbar*{ e^{-\st^2/(2\eff_i)} - \prod_{j = 0}^{\st - 1}
      e^{-j/\eff_i} }
      + \frac \st {\eff_i} \sum_{\substack{H \subset A
      \\ |H| \ge 1}} \, \vertbar*{ \,
        \prod_{j \in H} R(-j/\eff_i) \prod_{j \in A \setminus H}
      e^{-j/\eff_i} } \,.
      \label{eq:P-Q-diff-expanded}
    \end{split}
  \end{align}
  The left term in the second line of (\ref{eq:P-Q-diff-expanded}) is
  \begin{align}
    \frac \st {\eff_i} \vertbar*{
      e^{-\st^2/(2\eff_i)} - e^{-\st(\st - 1)/(2\eff_i)}
    }
    \le \frac {\st} {\eff_i} \vertbar*{ e^{-\st/(2\eff_i)} - 1 }
    \le \frac {\st^2} {\eff_i^2}
    \le \eff_i^{2\alpha - 2} \,.
    \label{eq:first-term-bound}
  \end{align}
  The first inequality follows by factoring out $e^{-\st(\st -
  1)/(2\eff_i)}$, which is less than or equal to 1.
  For the second inequality, we used that $|e^x - 1| \le 2|x|$ if $|x|
  \le 1$.

  We turn to the right term of (\ref{eq:P-Q-diff-expanded}), which we
  denote by $r$. Since $j$ and $\eff_i$ are nonnegative,
  $e^{-j/\eff_i}$ is less than or equal to $1$. Hence
  \begin{align*}
    r = \frac \st {\eff_i}
    \sum_{\substack{H \subset A \\ |H| \ge 1}} \, \vertbar*{ \,
      \prod_{j \in H} R(-j/\eff_i) \prod_{j \in A \setminus H}
      e^{-j/\eff_i}
    }
    \le \frac \st {\eff_i} \sum_{\substack{H \subset A \\ |H| \ge 1}} \,
    \vertbar*{ \,
      \prod_{j \in H} R(-j/\eff_i)
    } \,.
  \end{align*}
  If $|x| \le 1$, then $R(x) \le x^2$, so $R(-j/\eff_i) \le
  j^2/\eff_i^2$. But $j \le \st$, thus
  \begin{align}
    r \le \frac \st {\eff_i}
    \sum_{m = 1}^{\st} \sum_{\substack{H \subset A \\ |H| = m}}
    \paren* { \frac {\st^2} {\eff_i^2} }^{\!\! m}
    = \frac \st {\eff_i}
    \sum_{m = 1}^{\st} \binom \st m
    \paren* { \frac {\st^2} {\eff_i^2} }^{\!\! m}
    \le \frac \st {\eff_i}
    \sum_{m = 1}^\infty \, \paren* { \frac {\st^3} {\eff_i^2} }^{\!\! m}.
    \label{eq:before-geometric-sum}
  \end{align}
  The last inequality holds because $\binom {\st} m \le \st^m$. We move out
  one factor of $\st^3/\eff_i^2$ from the rightmost sum of
  (\ref{eq:before-geometric-sum}), so that the sum starts at $m = 0$
  instead of $m = 1$. Then, by the geometric sum formula
  \begin{align}
    r \le \frac {\st^4} {\eff_i^3}
    \sum_{m = 0}^\infty \, \paren*{ \frac {\st^3} {\eff_i^2} }^{\!\! m}
    = \frac {\st^4} {\eff_i^3} \;
    \frac 1 {1 - \st^3/\eff_i^2}
    \le \eff_i^{4\alpha - 3} \;
    \frac 1 {1 - \eff_i^{3\alpha - 2}} \,.
    \label{eq:geometric-sum}
  \end{align}
  The geometric sum converges because $\st \le \eff_i^{\alpha}$
  and $\alpha < 2/3$. Combining (\ref{eq:P-Q-diff-expanded}),
  (\ref{eq:first-term-bound}) and (\ref{eq:geometric-sum}) yields
  \begin{align*}
    \vertbar*{ \ryl_{\eff_i}(\st) - P(\stv_i = \st) }
    \le \eff_i^{2\alpha - 2} + \eff_i^{4\alpha - 3} \frac 1 {1 -
    \eff_i^{3\alpha - 2}} \,.
  \end{align*}
  By the assumption $\alpha \ge 1/2$, the first summand
  $\eff_i^{2\alpha - 2}$ is less than or equal to $\eff_i^{4\alpha -
  3}$. Since $\alpha$ is less than $2/3$, the fraction $1/\paren*{1 -
  \eff_i^{3\alpha - 2}}$ on the right
  in (\ref{eq:geometric-sum}) is decreasing in $\eff_i$. Thus the
  right hand side is $O\paren*{\eff_i^{4\alpha-3}}$.

  For part $(b)$, we distinguish the cases $\st > \eff_i$ and $\st
  \le \eff_i$. If $\st > \eff_i$, it follows from the
  definition in (\ref{eq:collision-probability}) that $P(\stv_i =
  \st) = 0$. So in this case
  \begin{align*}
    \vertbar*{\ryl_{\eff_i}(\st) - 0 }
    = \frac {\st} {\eff_i} e^{-\st^2/(2\eff_i)}
    \le \frac {\st} {\eff_i} e^{-(\st - 1)\st/(2\eff_i)} \,,
  \end{align*}
  since the exponential function is increasing. If $\st_i \le
  \eff_i$, we apply the inequality $1 + x \le e^x$ to
  (\ref{eq:collision-probability}) with $x=-j/\eff_i$, yielding
  \begin{align*}
    P(\stv_i = \st)
    \le \frac \st {\eff_i} \prod_{j = 0}^{\st - 1} e^{-j/\eff_i}
    = \frac {\st} {\eff_i} e^{-(\st - 1)\st/(2\eff_i)} \,.
  \end{align*}
  Since $\ryl_{\eff_i}(\st) \le (\st/\eff_i)e^{-(\st -
  1)\st/(2\eff_i)}$ as well, $P(\stv_i = \st)$ and
  $\ryl_{\eff_i}(\st)$ are both nonnegative real numbers bounded by
  $(\st/\eff_i)e^{-(\st - 1)\st/(2\eff_i)}$. Thus the absolute value
  of their difference is at most $(\st/\eff_i)e^{-(\st - 1)\st/(2\eff_i)}$, too.
\end{proof}

With Lemma \ref{lemma:pq-diff}, we can now complete the proof of
Theorem \ref{th:main}. To show that (\ref{eq:expectation-expr}) is
asymptotic to $T_{\idg_i, \tf_i}$, our strategy will be to split the
sums in (\ref{eq:expectation-expr}) at $\st_i = \big \lfloor
\eff_i^{9/16} \big \rfloor$ and apply Lemma \ref{lemma:pq-diff} part
$(a)$ for small $\st_i$ and part $(b)$ for large $\st_i$.

\begin{proof}[Proof of Theorem \ref{th:main}]
  By Lemma \ref{lemma:ryl-approximated-formula}, it suffices to show that
  $\E \paren*{ \min_{i = 1}^M \tv_i } \sim \rya_{\idg_i,\tf_i}(\nd)$.
  We do this by successively exchanging $\ryl_{\eff_i}(\st_i)$ for
  $P(\stv_i = \st_i)$ in (\ref{eq:def-ryl-approximated}), for every
  $1 \le i \le M$. To avoid complicated notation, the detailed
  computation is only given for $i = 1$.

  By splitting the first sum in (\ref{eq:def-ryl-approximated}) at $a
  \colonequals \big \lfloor \eff_1^{9/16} \big \rfloor$, we get
  \begin{align*}
    \rya_{\idg_i,\tf_i}(\nd)
    = &\sum_{\st_1 = 0}^a \ryl_{\eff_1}(\st_1) \, \cdots \!
    \sum_{\st_M = 0}^\infty \ryl_{\eff_M}(\st_M)
    \min_{i = 1}^M(\st_i \tf_i) \\
    &+ \sum_{\st_1 = a + 1}^\infty \ryl_{\eff_1}(\st_1) \, \cdots \!
    \sum_{\st_M = 0}^\infty \ryl_{\eff_M}(\st_M)
    \min_{i = 1}^M(\st_i \tf_i) \,.
  \end{align*}
  Applying Lemma \ref{lemma:pq-diff} part $(a)$ with $\alpha = 9/16$ to
  $\ryl_{\eff_1}(\st_1)$ to the first line yields
  \begin{align*}
    \rya_{\idg_i,\tf_i}&(\nd)
    = \sum_{\st_1 = 0}^a
    \paren*{ P(\stv_1 = \st_1) + O \paren*{ \eff_1^{-3/4} }}
    \cdots \!
    \sum_{\st_M = 0}^\infty \ryl_{\eff_M}(\st_M)
    \min_{i = 1}^M(\st_i \tf_i) \\
    &+ \sum_{\st_1 = a + 1}^\infty
    \paren*{
      P(\stv_1 = \st_1) +
      \paren*{ \ryl_{\eff_1}(\st_1) - P(\stv_1  = \st_1) }
    }
    \, \cdots \!
    \sum_{\st_M = 0}^\infty \ryl_{\eff_M}(\st_M)
    \min_{i = 1}^M(\st_i \tf_i) \,.
  \end{align*}
  By moving the $O$-term in the first line and $\ryl_{\eff_1}(\st_1) -
  P(\stv_1 - \st_1)$ in the second line under separate sums,
  \begin{align}
    \rya_{\idg_i,\tf_i}&(\nd) =
    \sum_{\st_1 = 0}^\infty
    P(\stv_1 = \st_1)
    \sum_{\st_2 = 0}^\infty \ryl_{\eff_2}(\st_2) \, \cdots \!
    \sum_{\st_M = 0}^\infty \ryl_{\eff_M}(\st_M)
    \min_{i = 1}^M(\st_i \tf_i) \nonumber \\
    & +\sum_{\st_1 = 0}^a
    O \paren*{\eff_1^{-3/4} } \,
    \sum_{\st_2 = 0}^\infty \ryl_{\eff_2}(\st_2) \, \cdots \!
    \sum_{\st_M = 0}^\infty \ryl_{\eff_M}(\st_M)
    \min_{i = 1}^M(\st_i \tf_i) \label{eq:ryl-approx-error}  \\
    &+ \sum_{\st_1 = a + 1}^\infty
    \paren*{\ryl_{\eff_1}(\st_1) - P(\stv_1 = \st_1) }
    \sum_{\st_2 = 0}^\infty \ryl_{\eff_2}(\st_2) \, \cdots \!
    \sum_{\st_M = 0}^\infty \ryl_{\eff_M}(\st_M)
    \min_{i = 1}^M(\st_i \tf_i) \,. \notag
  \end{align}
  Because our goal is to show that $\rya_{\idg_i,\tf_i}$ is
  asymptotically equivalent to the first line of
  (\ref{eq:ryl-approx-error}), we now argue that the second and third
  line are $o \paren*{\sqrt{\nd}}$. This suffices for the proof since
  $\rya_{\idg_i,\tf_i}(\nd) = \Theta \paren*{\sqrt\nd}$ by Lemma
  \ref{lemma:ryl-approximated-formula}.

  In the second line of (\ref{eq:ryl-approx-error}), the nested sum over
  $\st_2, \dots, \st_M$ is at most
  \begin{align*}
    \rya_{\idg_2, \dots, \idg_M, \tf_2, \dots, \tf_M}(\nd)
  \end{align*}
  and hence $O(\sqrt \nd)$ by
  Lemma \ref{lemma:ryl-approximated-formula}. Since $a = O
  \paren*{\nd^{9/16}}$, the second line is $O \paren*{\nd^{5/16}} =
  o(\sqrt \nd)$.
  Denote the third line of (\ref{eq:ryl-approx-error}) by $r$. By Lemma
  \ref{lemma:pq-diff} part $(b)$ and by removing $\st_1 \tf_1$ from the
  minimum, we get
  \begin{align*}
    \begin{split}
      |r| \le \sum_{\st_1 = a + 1}^\infty \frac {\st_1} {\eff_1}
      e^{-(\st_1 - 1)\st_1 / (2\eff_1)}
      \sum_{\st_2 = 0}^\infty \ryl_{\eff_2}(\st_2)
      \, \cdots \!
      \sum_{\st_M = 0}^\infty \ryl_{\eff_M}(\st_M)
      \min_{i = 2}^M(\st_i \tf_i) \,.
    \end{split}
  \end{align*}
  The nested sum over $\st_2, \dots, \st_M$ is independent
  from the first and at most $O \paren*{\sqrt {\nd}}$ by Lemma
  \ref{lemma:ryl-approximated-formula}.
  To bound the first sum from above, we may replace it by an integral
  from $a$ to $\infty$, because the function $(\st_1/\eff_1)
  e^{-(\st_1 - 1)\st_1 /
  (2\eff_1)}$ is decreasing in $\st_1$ for $\st_1 \ge a$. Thus
  \begin{align*}
    |r| \le \int_a^\infty \frac{\st_1}{\eff_1} e^{-(\st_1 - 1)\st_1 /
    (2\eff_1)} \, d\st_1
    \; O \paren*{\sqrt \nd} \,.
  \end{align*}
  Since $\st_1 - 1 \ge \st_1 / \eff_1$ for large enough $\eff_1$ and
  $e^x$ is increasing,
  \begin{align*}
    |r| \le \int_a^\infty (\st_1 - 1) e^{-(\st_1 - 1)^2/(2\eff_1)} \,
    d\st_1 \, O \paren*{\sqrt \nd}
    = \eff_1 e^{-(a - 1)^2/(2\eff_1)} O \paren*{\sqrt \nd} \,.
  \end{align*}
  But $a = \big\lfloor \eff_1^{9/16} \big\rfloor$ and $\eff_1 = \Theta(n)$, so
  \begin{align}
    |r| = O\paren*{\nd^{3/2} e^{-\eff_1^{1/8}/2} } = o(\sqrt \nd) \,.
    \label{eq:q-p-sum-bound}
  \end{align}
  The last estimate holds because the exponential decays faster than
  any polynomial as $n \to \infty$. Thus the third line of
  (\ref{eq:ryl-approx-error}) is $o(\sqrt{\nd})$, too. We conclude that
  \begin{align*}
    & \rya_{\idg_i,\tf_i}(\nd)
    \sim \sum_{\st_1 = 0}^\infty
    P(\stv_1 = \st_1)
    \sum_{\st_2 = 0}^\infty \ryl_{\eff_2}(\st_2) \, \cdots \!
    \sum_{\st_M = 0}^\infty \ryl_{\eff_M}(\st_M)
    \min_{i = 1}^M(\st_i \tf_i) \,.
  \end{align*}

  As already mentioned, the same argument can be applied to exchange
  $\ryl_{\eff_i}(\st_i)$ for $P(\stv_i = \st_i)$ for every $1 \le i
  \le M$. The only difference is that in the second and third
  line of (\ref{eq:ryl-approx-error}), some of the
  $\ryl_{\eff_i}(\st_i)$ are replaced by $P(\stv_i = \st_i)$. A
  little extra care is necessary to deal with this. Say we
  want to replace $\ryl_{\eff_{i'}}(\st_{i'})$ by $P(\stv_{i'} =
  \st_{i'})$ for some $2
  \le i' \le M$. The minimum on the right is bounded from above by
  $s_j \tf_j$ for
  some $j \ne i'$. Thus in line two and three of
  (\ref{eq:ryl-approx-error}), there $M - 1$ independent sums over
  $P(\stv_i = \st_i)$ or $\ryl_{\eff_i}(\st_i)$, where the $j$-th one
  has an additional factor $\st_j \tf_j$. The sums associated to a
  Rayleigh density $\ryl_{\eff_i}(\st_i)$ may be replaced by an
  integral via Lemma \ref{lemma:sum-integral}. For every $i \notin
  \{i', j\}$, we have $\sum_{s_i = 0}^\infty P(\stv_i = \st_i) = 1$
  or $\int_0^\infty \ryl_{\eff_i}(\st_i) \, d\st_i = 1$. We do not
  know whether $j$ is associated to a sum or integral. But in any
  case, $\sum_{s_j = 0}^\infty P(\stv_j = \st_j) \st_j \tf_j =
  O(\sqrt{\nd})$ by the above argument for $M = 1$ and Lemma
  \ref{lemma:ryl-approximated-formula}, or $\int_0^\infty
  \ryl_{\eff_j}(\st_j) \st_j \tf_j \, ds_j = \tf_j \sqrt{\pi \eff_j /
  2} = O(\sqrt\nd)$ by the formula for the mean of a
  Rayleigh-distributed random variable. Thus the second and third
  line of (\ref{eq:ryl-approx-error}) are still $o(\sqrt \nd)$.
\end{proof}

\section{The optimal k for one machine}
\label{sec:opt-k-M1}

Before we come to the proof of Theorem \ref{th:M1}, we rewrite
(\ref{eq:def-grt}) in a more convenient form, which might also be
useful for $M \ge 2$. Rather than summing over
all $d \in \paren*{\Z/\lcmki\Z}^\times$, we sum over all possible
greatest common divisors $d - 1$ can have with $\lcmki$. To do so, we
need to count how often each greatest common divisor occurs among all
$d \in \paren*{\Z/\lcmki\Z}^\times$. For $n \in \N_{\ge 1}$ and $g
\in \N$ dividing $n$, define
\begin{align}
  \psi(n, g) \colonequals
  \vertbar*{\curly*{x \in \paren*{\Z/n\Z}^\times :
  \gcd(x - 1, n) = g }} \,.
  \label{eq:def-psi}
\end{align}
As a convention, we set $\psi(1, 1) = 1$. By the Chinese Remainder
Theorem, $\psi$ is
multiplicative in following sense: For coprime $n, m \in \N$ and $g$
dividing $nm$
\begin{align}
  \psi(nm, g) = \psi(n, \gcd(g, n)) \cdot \psi(m, \gcd(g, m)) \,.
  \label{eq:psi-multiplicative}
\end{align}
Thus $\psi(\lcmki, g) = \prod_{p^e \mid \lcmki} \psi(p^e, \gcd(p^e,
g))$. The notation $\prod_{p^e \mid \lcmki}$ means the product over all
prime powers $p^e$ dividing $\lcmki$, such that $e \ge 1$ and $p^{e +
1} \nmid \lcmki$.

We can now write (\ref{eq:def-grt}) as\footnote{
  If $\psi(\lcmki, g) = 0$, we understand
  the respective summand to be excluded in (\ref{eq:grt-over-g}).
  In particular, all odd $g$ are excluded, since $\psi(2^e, 1) = 0$ for
  every $e \in \N_{\ge 1}$. This
  is relevant if $g$ is coprime to all the $k_i$, since then we would
divide by 0.}
\begin{align}
  \begin{split}
    \grt(k_1, \dots, k_M)
    &= \frac 1 {\phi(\lcmki)} \sum_{g \mid \lcmki}
    \paren*{ \sum_{i = 1}^M
    \frac {\gcd(g, 2k_i) - 1} {\log_2^2 2k_i}}^{\!\!-1/2}
    \psi(\lcmki, g) \\
    &= \frac 1 {\phi(\lcmki)} \sum_{g \mid \lcmki}
    \paren*{ \sum_{i = 1}^M
    \frac {\gcd(g, 2k_i) - 1} {\log_2^2 2k_i}}^{\!\!-1/2}
    \prod_{p^e \mid \lcmki} \psi(p^e, \gcd(p^e, g)) \,.
  \end{split}
  \label{eq:grt-over-g}
\end{align}
To compute $\psi(n, g)$, it suffices to consider the case where $n$
is a prime power by (\ref{eq:psi-multiplicative}).
\begin{lemma}
  Let $p$ be a prime number, $e \in \N_{\ge 1}$ and $g$ be a divisor
  of $p^e$. Then
  \begin{align}
    \psi \paren*{p^e, g} =
    \begin{cases}
      \phi(p^e) - p^{e - 1} & \quad \text{if } g = 1 \\
      \phi(p^e / g) & \quad \text{if } g > 1
    \end{cases} \,.
    \label{eq:psi-prime-power}
  \end{align}
  \label{lemma:psi-prime-power}
\end{lemma}

\begin{proof}
  First assume $g = 1$. There are $\phi(p^e)$ units
  in $\Z/p^e\Z$ and they are precisely the congruence classes not divisible
  by $p$. But we must exclude the units equivalent to 1 modulo
  $p$ due to the condition $\gcd(x - 1, p^e) = g = 1$ in
  (\ref{eq:def-psi}). Since there are $p^{e - 1}$ such units, it follows that
  $\psi(p^e, g) = \phi(p^e) - p^{e - 1}$.

  If $g > 1$, the condition $\gcd(x - 1,
  p^e) = g$ implies $\gcd((x - 1)/g, p^e/g) = 1$. This yields a map
  \begin{align}
    \begin{split}
      \curly*{x \in \paren*{\Z/p^e\Z}^\times : \gcd(x - 1, p^e) = g }
      &\to \paren*{\Z/(p^e/g)\Z}^\times \\
      x &\mapsto (x - 1)/g \,.
    \end{split}
    \label{eq:bij-units}
  \end{align}
  It is easy to see that this map is well defined and injective. To prove
  that it is surjective, let $y \in \Z/(p^e/g)\Z$. Identifying the
  congruence class $y$ with its representative in $[0, p^e/g - 1]$,
  we consider $gy + 1$ as a preimage. Since $p \mid g$ by assumption,
  it follows that $gy + 1
  \in \paren*{\Z/p^e\Z}^\times$. We also have $\gcd(gy, p^e) = g$,
  because $p$ does not divide $y$. Thus the map in
  (\ref{eq:bij-units}) is bijective, so $\psi(p^e, g) = \phi(p^e/g)$.
\end{proof}

\begin{proof}[Proof of Theorem \ref{th:M1}]
  We write $k \colonequals k_1$ in this proof. Then $\lcmki =
  2k$. Substituting $M = 1$ into (\ref{eq:grt-over-g}) yields
  \begin{align}
    \grt(k) = \frac {\log_2 2k} {\phi(2k)}
    \sum_{g \mid 2k} \frac 1 {\sqrt{g - 1}}
    \prod_{p^e \mid 2k} \psi(p^e, g) \,.
    \label{eq:grt-M1}
  \end{align}
  Since $\grt(1) = 1$, our aim is to show that $\grt(k) > 1$ for all $k \ge
  2$. To derive a lower bound, we only consider the summand where $g
  = 2$. Let $e_2 \in \N$ such that $2^{e_2} \mid 2k$ and $2^{e_2 + 1}
  \nmid 2k$ (so $e_2 \ge 1$). Then $\psi(2^{e_2}, g) = \phi(2^{e_2 -
  1})$ and $\psi(p^e, g) = \phi(p^e) - p^{e - 1}$ for every prime $p$
  not equal to 2,  by Lemma \ref{lemma:psi-prime-power}. Thus
  \begin{align*}
    \grt(k)
    \ge \frac {\log_2 2k} {\phi(2k)} \prod_{p^e \mid 2k} \psi(p^e, g)
    = \frac {\log_2 2k} {\phi(2k)} \, \phi(2^{e_2 - 1})
    \prod_{p^e \mid 2k, p \ne 2} (\phi(p^e) - p^{e - 1}) \,.
  \end{align*}
  Factoring out $\phi(p^e)$ from the product on the right for every
  $p$ and using that $\phi(p^e) =
  p^e - p^{e - 1}$ yields
  \begin{align}
    \grt(k) \ge \frac {\log_2 2k} {\phi(2k)} \phi(k)
    \prod_{p \mid k, p \ne 2} \paren*{ 1 - \frac 1 {p - 1} }
    \ge \frac {\log_2 2k} 2
    \prod_{p \mid k, p \ne 2} \paren*{ 1 - \frac 1 {p - 1} } \, .
    \label{eq:grt-M1-prod}
  \end{align}
  Let
  \begin{align*}
    \mathcal P \colonequals \prod_{p \mid k, p \ne 2} \paren*{ 1 -
    \frac 1 {p - 1} } = \exp \paren*{
      \sum_{p \mid k, p \ne 2}
      \log \paren*{1 - \frac 1 {p - 1}}
    } \,.
  \end{align*}
  Since $\log \paren*{1 + x} \ge x - x^2/2 + (5/9)x^3$ for $x \in [-1/2, 0]$,
  \begin{align}
    \mathcal{P}
    \ge \exp \paren* {
      - \smashoperator[l]{\sum_{p \mid k, p \ne 2}} \frac 1 {p - 1}
      - \smashoperator[l]{\sum_{p \mid k, p \ne 2}} \frac {1/2}
      { \paren*{p - 1}^2}
      - \smashoperator[l]{\sum_{p \mid k, p \ne 2}} \frac {5/9}
      { \paren*{p - 1}^3}
    }.
    \label{eq:log-taylor}
  \end{align}
  Let $\omega(k)$ be the number of distinct prime divisors of $k$.
  Denote by $p_i$ the $i$-th prime number (so $p_1 = 2, \, p_2 = 3,
  \, \dots$). Because $1/(p - 1)$ is decreasing in $p$, we may
  replace the prime divisors of $k$ by $p_2, \dots, p_{\omega_2(k) +
  1}$ in the sums in (\ref{eq:log-taylor}). If $2 \mid k$, one
  additional prime is included this way, since the sums in
  (\ref{eq:log-taylor}) range only over primes not equal to 2.
  To derive a lower bound, however, this is acceptable. Thus
  \begin{align*}
    \mathcal{P}
    \ge \exp \paren*{
      - \smashoperator[l]{\sum_{i = 2}^{\omega(k) + 1}} \!
      \frac 1 {p_i - 1}
      - \smashoperator[l]{\sum_{i = 2}^{\omega(k) + 1}} \!
      \frac {1/2} {\paren*{p_i - 1}^2}
      - \smashoperator[l]{\sum_{i = 2}^{\omega(k) + 1}} \!
      \frac {5/9} {\paren*{p_i - 1}^3}
    } \,.
  \end{align*}
  Since $p_i - 1\ge p_{i - 1}$,
  \begin{align*}
    \mathcal{P}
    \ge \exp \paren*{
      -\sum_{i = 1}^{\omega(k)} \frac 1 {p_i}
      - \frac 1 2 \sum_{i = 1}^{\omega(k)} \frac 1 {p_i^2}
      - \frac 5 9 \sum_{i = 1}^{\omega(k)} \frac 1 {p_i^3}
    } \,.
  \end{align*}
  Using that $\sum_{i = 1}^\infty 1/p_i^2 \approx 0.452$ and $\sum_{i
  = 1}^\infty 1/p_i^3 \approx 0.125$, we can bound this by
  \begin{align}
    \mathcal{P}
    > \exp \paren*{
      -\sum_{i = 1}^{\omega(k)} \frac 1 {p_i}
      - \frac 3 {10}
    }.
    \label{eq:prod-p-1-bound}
  \end{align}
  If $\omega(k) \le 2$, then by (\ref{eq:prod-p-1-bound}) $\mathcal
  P$ is bounded by $e^{-17/15}$ from below, since $1/2 + 1/3 + 3/10 =
  17/15$. Inserting this into (\ref{eq:grt-M1-prod}) yields $\grt(k) >
  \paren* {\log_2 2k} / (2e^{17/15})$, which shows that $\grt(k) > 1$
  for $k \ge 38$.

  Now to the case $\omega(k) \ge 3$. By Rosser's theorem \cite{ros39}, $p_i >
  i \log i$. Thus
  \begin{align}
    \sum_{i = 1}^{\omega(k)} \frac 1 {p_i}
    < 5/6 + \sum_{i = 3}^{\omega(k)} \frac 1 {i \log i}
    \le  5/6 +
    \smashoperator[lr]{\int_2^{\omega(k)}} \frac 1 {x \log x} \, dx
    < \log \log \omega(k) + 6/5 \, .
    \label{eq:rosser}
  \end{align}
  Plugging (\ref{eq:rosser}) into (\ref{eq:prod-p-1-bound}), we get
  $\mathcal P > e^{-\log \log \omega(k) - 3/2}$. But $\omega(k) <
  \log k$ by the assumption $\omega(k) \ge 3$, so $\mathcal P >
  e^{-3/2}/(\log \log k)$. Applying this bound to
  (\ref{eq:grt-M1-prod}), we obtain
  \begin{align}
    \grt(k) > \frac {\log_2 2k} {2 e^{3/2} \log \log k} \, .
    \label{eq:ek-final-bound}
  \end{align}
  If $k \ge 2.1 \cdot 10^7$, the right side of
  (\ref{eq:ek-final-bound}) is greater than 1. For all $2 \le k <
  2.1 \cdot 10^7$, the inequality
  $\grt(k) > 1$ was verified numerically, with the code available at
  \url{https://github.com/finn-rudolph/rhok-formula}.
\end{proof}

\section{Table of values for G}
\label{sec:table-of-values}

Values of $\grt$ relative to $\grt(1, 1, \dots, 1)$ are displayed in
Figure \ref{fig:M1} for $M = 1$ and in Figure
\ref{fig:M2} for $M = 2$. For $M = 1$, the running time generally
increases with $k$ and has local maxima at prime $k$ (except for
$k = 2, 3$). Local minima occur at numbers with many small distinct
prime factors.

\begin{figure}[!htb]
  \centering

  \begin{tikzpicture}
    \begin{axis}[
        xlabel = {$k$},
        ylabel = {$G(k)$},
        xmin = 1,
        xmax = 64,
        xtick = {1, 8, 16, 24, 32, 40, 48, 56, 64},
        width = \textwidth,
        height = 240pt,
        legend pos = north west,
        legend style = {draw = none},
        legend cell align = left,
      ]

      \addplot[mark=square]
      coordinates{
        (1 ,   1)
        (2 ,   1.5773502691896257)
        (3 ,   1.8704964374506134)
        (4 ,   2.2164860566491402)
        (5 ,   2.7682734124061352)
        (6 ,   2.0847230997701387)
        (7 ,   3.3487908119117074)
        (8 ,   2.8954319505678203)
        (9 ,   2.8751374461044583)
        (10,   2.860468041948317)
        (11,   4.110801232720638)
        (12,   2.49867142236409)
        (13,   4.387077070265018)
        (14,   3.347742428949697)
        (15,   2.981346656484012)
        (16,   3.594729442047154)
        (17,   4.824847264057605)
        (18,   2.86115159715859)
        (19,   5.004306760955722)
        (20,   3.300312143156222)
        (21,   3.4464061409132047)
        (22,   3.976341228294828)
        (23,   5.3099184126352545)
        (24,   2.9799685734007593)
        (25,   4.649462480147749)
        (26,   4.201908355307487)
        (27,   3.9343331208985237)
        (28,   3.78895160584554)
        (29,   5.676478407355586)
        (30,   2.9006003880189937)
        (31,   5.78113499064364)
        (32,   4.303622115889641)
        (33,   4.03982369338442)
        (34,   4.557322202653078)
        (35,   4.4996861521754505)
        (36,   3.200039415360687)
        (37,   6.057156370621868)
        (38,   4.702402609745771)
        (39,   4.251370145477844)
        (40,   3.840056366646475)
        (41,   6.216273071182121)
        (42,   3.293671983794555)
        (43,   6.289854314772301)
        (44,   4.407589658164336)
        (45,   3.7566944358569265)
        (46,   4.948752651453858)
        (47,   6.426873424506402)
        (48,   3.488462083590686)
        (49,   5.79033355040268)
        (50,   4.344186892962921)
        (51,   4.583494417105624)
        (52,   4.627075627502184)
        (53,   6.611163865172371)
        (54,   3.704993730277928)
        (55,   5.213244247322634)
        (56,   4.350804017518562)
        (57,   4.718699359650112)
        (58,   5.243191113483301)
        (59,   6.77494748945913)
        (60,   3.1790610972370525)
        (61,   6.825726165781633)
        (62,   5.327072193434716)
        (63,   4.248127467002441)
        (64,   5.016817975614528)
      };
    \end{axis}
  \end{tikzpicture}
  \vspace{-2.2em}
  \caption{Values of $\grt(k)$ for $M = 1$.}
  \label{fig:M1}
\end{figure}
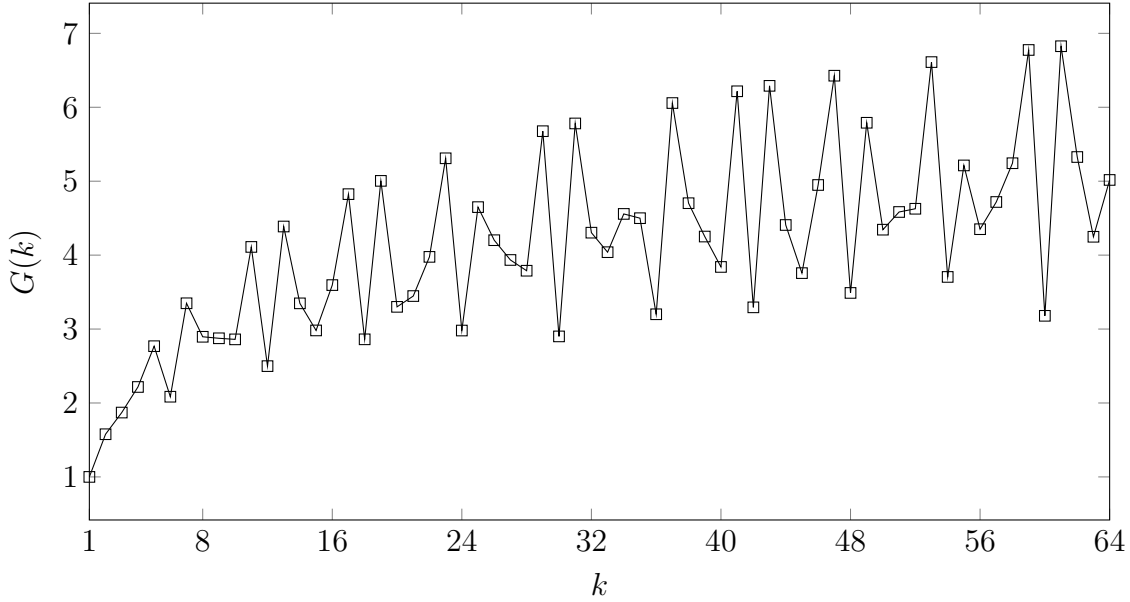

\begin{figure}[!htb]
  \centering
  \footnotesize
  \newcolumntype{C}{m{1.7em}}

  \begin{tabular}{c|CCCCCCCCCCCCCC}
    $k_1 \backslash k_2$ & 1 & 2 & 3 & 4 & 5 & 6 & 7 & 8 & 9 & 10 &
    11 & 12 & 13 & 14 \\
    \hline
    1 & 1.00 & 1.17 & 1.19 & 1.24 & 1.28 & 1.22 & 1.31 & 1.28 & 1.27
    & 1.28 & 1.34 & 1.25 & 1.35 & 1.31 \\
    2 &  & 1.58 & 1.57 & 1.80 & 1.80 & 1.70 & 1.90 & 1.93 & 1.78 &
    1.86 & 1.99 & 1.80 & 2.01 & 1.93  \\
    3 &  &  & 1.87 & 1.82 & 1.98 & 1.92 & 2.12 & 2.00 & 2.20 & 1.98 &
    2.25 & 2.02 & 2.29 & 2.09 \\
    4 &  &  &  & 2.22 & 2.21 & 2.02 & 2.39 & 2.48 & 2.18 & 2.33 &
    2.57 & 2.23 & 2.62 & 2.48 \\
    5 &  &  &  &  & 2.77 & 2.10 & 2.75 & 2.52 & 2.47 & 2.74 & 3.01 &
    2.28 & 3.09 & 2.69 \\
    6 &  &  &  &  &  & 2.08 & 2.26 & 2.25 & 2.31 & 2.17 & 2.42 & 2.25
    & 2.47 & 2.31 \\
    7 &  &  &  &  &  &  & 3.35 & 2.78 & 2.73 & 2.75 & 3.42 & 2.49 &
    3.52 & 3.26  \\
    8 &  &  &  &  &  &  &  & 2.90 & 2.50 & 2.69 & 3.06 & 2.54 & 3.14
    & 2.92 \\
    9 &  &  &  &  &  &  &  &  & 2.88 & 2.47 & 3.00 & 2.52 & 3.09 & 2.68 \\
    10 &  &  &  &  &  &  &  &  &  & 2.86 & 3.02 & 2.40 & 3.10 & 2.82 \\
    11 &  &  &  &  &  &  &  &  &  &  & 4.11 & 2.71 & 4.00 & 3.34 \\
    12 &  &  &  &  &  &  &  &  &  &  &  & 2.50 & 2.78 & 2.58 \\
    13 &  &  &  &  &  &  &  &  &  &  &  &  & 4.39 & 3.45 \\
    14 &  &  &  &  &  &  &  &  &  &  &  &  &  & 3.35 \\
  \end{tabular}
  \caption{Values of $G(k_1, k_2)$ relative to $G(1, 1)$.}
  \label{fig:M2}
\end{figure}

For $M = 2$, the running time still generally increases with the
$k_i$. However, there is a notable difference to $M = 1$. While for $M =
1$, the running time is strictly increasing for prime $k$, this need
not be the case for fixed $k_1$ and $k_2$ varying over the primes.
Consider $k_1 = k_2 = 2$ and $k_1 = 2, k_2 = 3$ as an example. The
running time for $k_1 = k_2 = 2$ is higher, even though both 2 and 3
are prime. The same happens for the primes 5, 7 and 11, 13. This
suggests that it is better to ``bet on'' multiple primes than to
focus on just one prime.

\section{Final Remarks}
\label{sec:final-remarks}

\subsection{Optimal iteration maps for more than one machine}
\label{sec:opt-k-general-M}

The most interesting open question is certainly: Given some $M \ge
2$, which $k_1, \dots, k_M$ minimize $G(k_1, \dots, k_M)$?
Considering the range of $k_i$ displayed in Figure \ref{fig:M2}, it
is best to choose $k_1 = k_2 = 1$ for $M = 2$. Moreover, since the
running time generally increases with the $k_i$, the choice $k_1 =
k_2 = 1$ seems to be optimal. This pattern holds for $M \ge 3$ as well.
To explore the behavior of $G$ for larger $M$, I used the program
available at \url{https://github.com/finn-rudolph/rhok-formula} and
computed values of $G$ for $2 \le M \le 10$ and $1 \le k_i \le
k_{\max}$, as shown in the table below.
\begin{center}
  \smallskip
  \begin{tabular}{c|ccccccccc}
    $M$ & 2 & 3 & 4 & 5 & 6 & 7 & 8 & 9 & 10 \\
    $k_{\max }$ & 3000 & 400 & 120 & 48 & 30 & 24 & 22 & 19 & 16
  \end{tabular}
  \smallskip
\end{center}
For every $2 \le M \le 10$, the running time is minimal at $k_1 =
\cdots = k_M = 1$, among all choices of $k_i$ with $1 \le k_i \le
k_{\max}$ for every $i$. Also, it generally increases with larger
$k_i$, like in Figures
\ref{fig:M1} and \ref{fig:M2}, suggesting that $k_i = 1$ for all $1
\le i \le M$ is optimal for every $M \in \N_{\ge 1}$.

\subsection{The time per iteration}
\label{sec:time-per-it}

As briefly noted in §\ref{sec:general-formulation}, we assume that the
duration of each iteration of the $i$-th machine is proportional to
$\log_2 2k_i$. This assumption was used by setting $\lambda_i = \log_2 2k_i$ in
the derivation of (\ref{eq:main}) from Theorem \ref{th:main}. To
justify it, the optimization described by Brent \cite{bre80} must be
applied. Then, the $\gcd$ computation happens in
batches, which only scales the total running time by a constant and
does not impact the per-iteration cost. The work per iteration thus
only consists of the computation of the power $x^{2k_i}$, a few
additions and one multiplication. The power requires around
$\log_2 2k_i$ steps using binary exponentiation and the other
operations are negligible.

\subsection{The random map heuristic}
\label{sec:rma}

Numerical results supporting Heuristic \ref{heu:rma} are given in
\cite{bp81}. There are few rigorous results about the
apparently random behavior of $\itfn_i$. Bridy and Garton remark in
\cite{bg18} that the cycle structure of maps of the form $x \mapsto
x^k + a$ cannot exactly match the cycle structure of random maps,
because the number of cycles of length $m$ is inherently bounded by
$m^{-1} \deg f_i^m$. However, they show that it is ``as
random as possible'' under this constraint. In \cite{ba91}, Bach
proves a rigorous estimate for the success probability of the rho
method after a certain number of steps.

\printbibliography

\end{document}